\documentclass{article}
\usepackage{graphicx}
\graphicspath{images/}

\usepackage{blindtext}

\usepackage[utf8]{inputenc}
\usepackage{latexsym, amssymb, latexsym, amsfonts, amsmath, esint, graphicx, caption, subcaption, dsfont, units, fullpage, float, setspace, cases, polynom, multirow , stmaryrd , textcomp, tikz, mathtools, amsthm, epstopdf, epsfig, hyperref}
\usepackage[first=1, last=10]{lcg}
\usepackage{listings,enumerate,geometry,csquotes,url}
\usepackage{etoolbox}
\usepackage{pdfpages,xspace,xparse}
\usepackage{import,pst-plot}
\usepackage[numbers,sort]{natbib}
\usepackage[toc,page]{appendix}

\newcommand{\beq}{\begin{equation}}
\newcommand{\eeq}{\end{equation}}
\newcommand{\ba}{\begin{array}}
\newcommand{\ea}{\end{array}}
\newcommand{\bea}{\begin{eqnarray*}}
\newcommand{\eea}{\end{eqnarray*}}
\newcommand{\bc}{\begin{center}}
\newcommand{\ec}{\end{center}}
\newcommand{\bt}{\begin{table}}
\newcommand{\et}{\end{table}}

\newcommand{\la}[1]{\label{#1}}

\newcommand{\no}{\noindent}

\newcommand{\rf}[1]{(\ref{#1})}
\newcommand{\beqno}{\begin{displaymath}}
\newcommand{\eeqno}{\end{displaymath}}

\newcommand{\been}{\begin{enumerate}}
\newcommand{\een}{\end{enumerate}}

\newcommand{\C}{\mathbb{C}}

\newcommand{\Arg}{\mathrm{Arg}}

\renewcommand{\And}{\qquad \text{ and } \qquad}

\newlength{\myheight}
\newlength{\mylength}

\newcounter{saveeqn}

\newtheorem{theorem}{Theorem}

\def\XXint#1#2#3{{\setbox0=\hbox{$#1{#2#3}{\int}$}
     \vcenter{\hbox{$#2#3$}}\kern-.5\wd0}}

\hypersetup{
    colorlinks=true,
    linkcolor=blue,
    citecolor=red
}
\usepackage{subfiles} 

\title{Solving the heat equation with variable thermal conductivity}
\author{
Matthew Farkas, Bernard Deconinck\\
Department of Applied Mathematics\\
University of Washington\\
Seattle, WA 98195-2420
}
\date{\today}

\begin{document}

\maketitle

\begin{abstract}
We consider the heat equation with spatially variable thermal conductivity and homogeneous Dirichlet boundary conditions. Using the Method of Fokas or Unified Transform Method, we derive solution representations as the limit of solutions of constant-coefficient interface problems where the number of subdomains and interfaces becomes unbounded. This produces an explicit representation of the solution, from which we can compute the solution and determine its properties. Using this solution expression, we can find the eigenvalues of the corresponding variable-coefficient eigenvalue problem as roots of a transcendental function. We can write the eigenfunctions explicitly in terms of the eigenvalues. The heat equation is the first example of more general variable-coefficient second-order initial-boundary value problems that can be solved using this approach. 
\end{abstract}

\setcounter{tocdepth}{4}
\setcounter{secnumdepth}{4}


\section{Introduction}

The Method of Fokas or Unified Transform Method (UTM) can be used to solve constant-coefficient Initial-Boundary Value Problems (IBVPs) \cite{deconinctrogdonvasan, JC_fokas_book}. The purpose of this paper is to demonstrate a method to generalize the UTM to solve variable-coefficient IBVPs. In \cite{treharne,vc_fokas}, Fokas and Treharne use a Lax-Pair approach to analyze some variable-coefficient IBVPs. This method reduces the problem of solving a {\em partial} differential equation (PDE) to that of solving an {\em ordinary} differential equation (ODE) by writing the solution of the PDE as an integral over the solutions to a non-autonomous ODE, but it does not provide an explicit representation of the solution. This approach, along with separation of variables, is useful if the associated ODE is a second-order, self-adjoint problem on a finite domain, for which we have standard Sturm-Liouville theory. However, it does not generalize well to problems that are not self adjoint, of higher order, or are posed on an infinite domain.

In our approach to variable-coefficient IBVPs, we break the domain into $N$ subdomains. The variable coefficients are approximated on each subdomain by constants, resulting in a constant-coefficient interface problem. We solve this problem using the UTM as shown in \cite{interface_heat, interface_schrodinger, interface_heat_ring, interface_maps, interface_kdv, interface_dispersive}. Cramer's rule gives the solution in each part as a ratio of determinants. Using an explicit expression for these determinants, we take the limit as $N$ goes to infinity. Obtaining the explicit expressions for the determinants and calculating the limit are both non-trivial steps. Finally, we obtain an explicit (albeit complicated) solution to the original variable-coefficient IBVP, useful for computating the solution, for instance. Further, our solution representation characterizes the eigenvalues of the spectral problem obtained after separation of variables and it gives the eigenfunctions explicitly in terms of these eigenvalues. 

Since the UTM is applicable to non-constant boundary conditions, higher-order and  non-self-adjoint problems, we expect our method to generalize similarly. Indeed, we have found explicit solutions for general, second-order IBVPs with spatially-variable coefficients and with general boundary conditions in terms of sums and integrals over known quantities. This will be reported in \cite{farkasdeconinck2}. Although our approach is entirely different, some of our notation has been inspired by \cite{PT}.


\section{The heat equation with homogeneous, Dirichlet boundary conditions}

\label{sec:statement}
Consider the heat equation on the finite interval, $x\in(0,1)$, with spatially-variable thermal conductivity $\sigma^2(x)$, without forcing and with homogeneous, Dirichlet boundary conditions:

\begin{subequations}\label{eqn:IBVP}
\begin{align} \label{eqn:IBVP_PDE}
q_t &= \left(\sigma^2(x) q_{x}\right)_x, && x\in(0,1), \quad t>0, \\ \label{eqn:IBVP_IC}
q(x,0)&=q_0(x), && x\in(0,1), \\ \label{eqn:IBVP_BC}
q(0,t)&= q(1,t)= 0, && t>0,
\end{align}
\end{subequations}

\no where the $x$ index denotes partial differentiation. 

\begin{theorem}
If $\sigma(x)>0$ is absolutely continuous and if $\sigma'(x)/\sigma(x)$ and $q_0(x)$ are absolutely integrable, then the IBVP \rf{eqn:IBVP} has the solution

\beq \label{eqn:q_sol}
q(x,t) = \frac{1}{i\pi} \int_{\partial \Omega} \frac{\Phi(k,x)}{\Delta (k)}  e^{-k^2t}\, dk, 
\eeq

\no where $\Omega=\{k\in\C: \pi/4 < \Arg(k) < 3\pi/4 \text{ and } |k|>r\}$ for some $r>0$ as shown in Figure~\ref{fig:Omega}{\textcolor{blue}a}, with

\begin{figure}[tb]
    \centering
    \begin{tabular}{ccc}
    \includegraphics[scale=0.5]{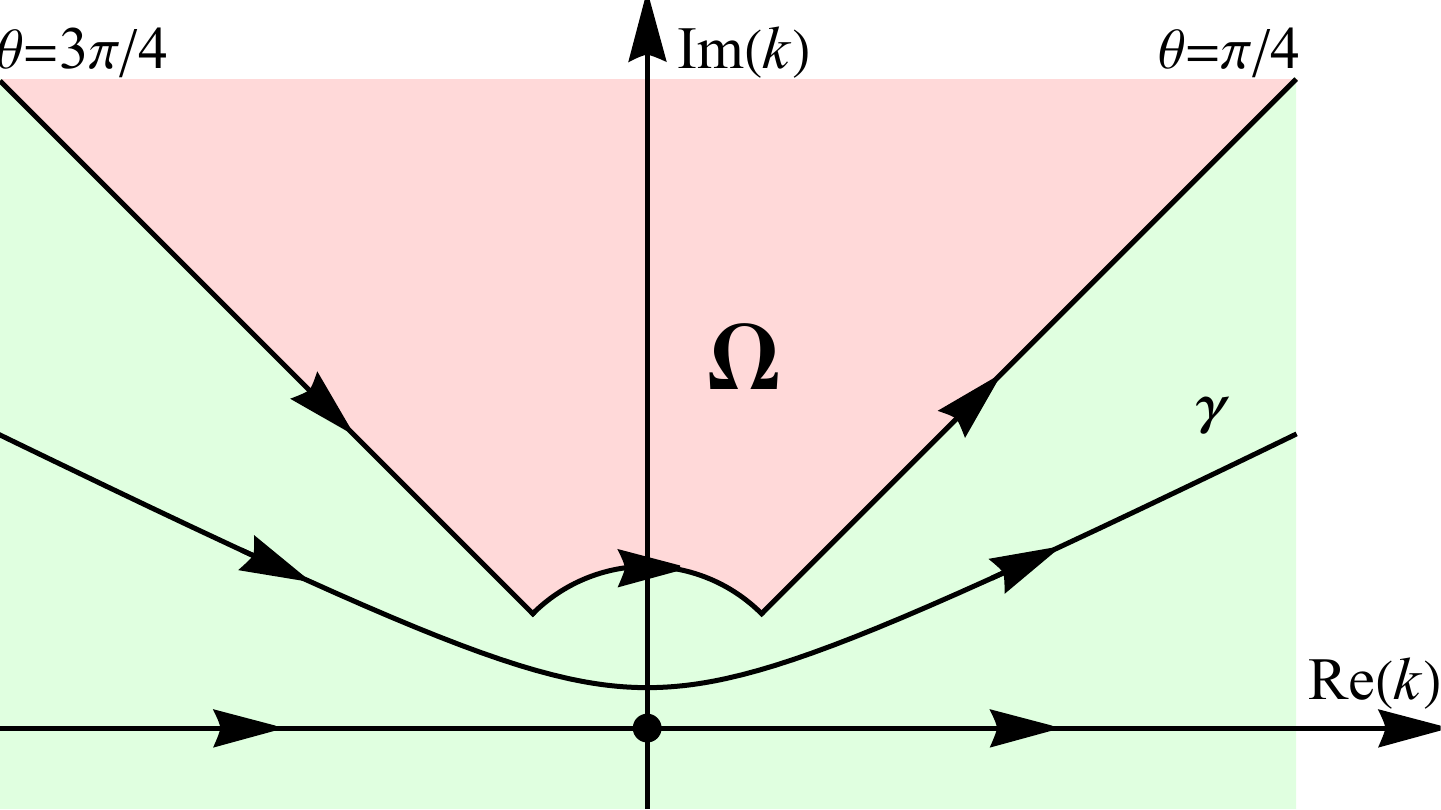} 
    & \hspace*{0.1in} &
    \includegraphics[scale=0.5]{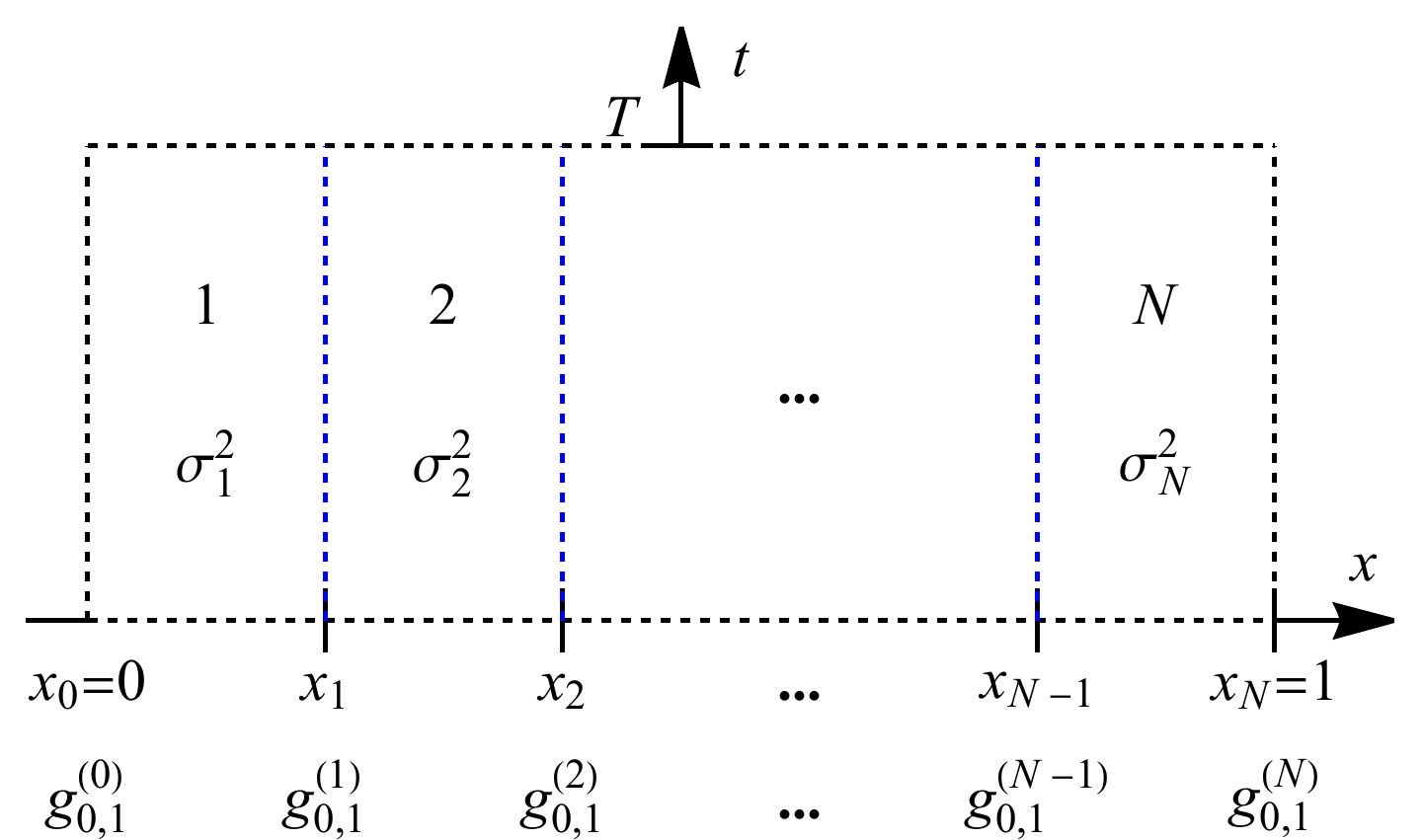} \\
    (a) & \hspace*{0.1in} & (b)
    \end{tabular}
    \caption{(a) The region $\Omega$ and the deformed contour $\gamma$. 
    (b) A partition of the finite interval $(0,1)$.}
    \label{fig:Omega}
\end{figure}

\beq\label{deltaphi}
\Delta(k) = \sum_{n=0}^\infty  \mathcal S_n^{(0,1)}(k) \And \Phi(k,x) = \int_0^1 \frac{\Psi(k,x,y)q_0(y)} {\sqrt{\sigma(x)\sigma(y)}} dy,
\eeq

\no where, for $0<y<x<1$,

\beq
\Psi(k,x,y) = \left( \sum_{n=0}^\infty  \mathcal S_{n}^{(0,y)}(k) \right) \left( \sum_{n=0}^\infty  \mathcal S_n^{(x,1)}(k)\right) =  \sum_{n=0}^\infty \sum_{\ell=0}^n \mathcal S_{n-\ell}^{(0,y)}(k) \mathcal S_\ell^{(x,1)}(k) ,
\eeq

\no and $\Psi(k,x,y) = \Psi(k,y,x)$ for $0<x<y< 1$. Here

\beq \label{eqn:Sn}
\mathcal S_n^{(a,b)}(k) = \frac{1}{2^n}\int_{a=y_0 \leq y_1 \leq \cdots  \leq y_n \leq y_{n+1}=b} \left(\prod_{p=1}^n \frac{\sigma'(y_p)}{\sigma(y_p)} \right) \sin\left(k \sum_{p=0}^{n} (-1)^p \int_{y_p}^{y_{p+1}}\frac{d\xi}{\sigma(\xi)}\right) \, dy_1 \cdots dy_n.
\eeq
\no Note that this reduces to the solution given in \cite{JC_fokas_book} for constant $\sigma(x)$.
\end{theorem}

\section{Derivation}

\label{sec:derivation}
We form a partition $x_j$ of the interval $\left[0,1\right]$, such that $x_0=0$ and $x_N=1$, see Figure~\ref{fig:Omega}{\textcolor{blue}b}. On each subdomain, we replace \rf{eqn:IBVP_PDE} with a constant-coefficient PDE with $\sigma(x)$ replaced by $\sigma_j$ such that $\sigma_j \to \sigma(x_j)$ as $N\to\infty$, with the corresponding initial condition. At each interface, we require continuity of the solution and a jump discontinuity in the derivative of the solution, consistent with the evolution equation, {\em i.e.},

\begin{subequations}
\begin{align} \label{eqn:interface_pde}
q_t^{(j)} &= \sigma_j^2 q_{xx}^{(j)}, & q^{(j)}(x,0) &= q_0(x), && x \in (x_{j-1}, x_{j}), \quad t>0,\quad j=1,\cdots, N, \\ \label{eqn:interface_continuity} q^{(j)}(x_{j}, t) &= q^{(j+1)}(x_{j}, t), & \sigma_{j}^2q_x^{(j)} (x_{j},t) &= \sigma_{j+1}^2 q_x^{(j+1)}(x_{j}, t), && t>0, \quad j=1,\cdots, N-1,
\\ \label{eqn:interface_bc} q^{(1)}(0,t) &= 0, &  q^{(N)}(1,t) &= 0, && t>0.
\end{align}
\end{subequations}

\no Note that the derivative jump \rf{eqn:interface_continuity} can be derived by integrating \rf{eqn:IBVP_PDE} over a small interval containing $x_j$. 

We follow \cite{interface_heat, interface_schrodinger, interface_heat_ring, interface_maps, interface_kdv, interface_dispersive}. Defining for $j=1, \ldots, N$ and $m=0,1$, 

\begin{gather}
\hat q_0^{(j)}(k)= \int_{x_{j-1}}^{x_{j}} e^{-ik y} q_0(y) dy,~~~~ 
\hat q^{(j)}(k,t)= \int_{x_{j-1}}^{x_{j}}e^{-ik y}q^{(j)}(y,t) dy,~~~~ 
g_m^{(j)}(W, t) = \int_0^t e^{Ws} q_{mx}^{(j)}(x_{j}, s) ds, \\  
g_m^{(0)}(W,t)=\int_0^t e^{W s} q_{mx}^{(1)}(0,s) ds,  ~~~~~~g_0^{(0)}(W,t) = g_0^{(N)}(W,t) = 0, \label{eqn:BC}
\end{gather}

\no where these last equations originate from the boundary conditions \rf{eqn:interface_bc}.
Introducing the dispersion relation ${w_j=\sigma^2_j \kappa^2}$, we obtain the {\em local relations} corresponding to \rf{eqn:interface_pde} defined in each subdomain $D_j = (x_{j-1}, x_{j})\times (0,T)$, $j=1, \ldots, N$. Integrating over the boundary of each subdomain and  using Green's theorem, we find the {\em global relations}. Changing variables $\kappa  = \nu_j(k) = k/\sigma_j$, these are
\begin{align}  \label{eqn:globalrelations} \nonumber
e^{k^2t}\hat q^{(j)}(\nu_j, t) &= \hat q_0^{(j)}\left(\nu_j\right) + e^{-i\nu_j x_{j}}\left( \sigma_{j}^2 g_1^{(j)}(k^2,t) +i\sigma_j k g_0^{(j)}(k^2,t)\right)
\\&~~
- e^{- i\nu_j x_{j-1}}\left( \sigma_{j-1}^2 g_1^{(j-1)}(k^2,t) + i\sigma_j k g_0^{(j-1)}(k^2,t)\right), ~~~~j=1, \ldots, N.
\end{align}

These relations are valid for $k\in \C$, since all integrals are over bounded domains. Letting $k\mapsto -k$, (and $\nu_j\mapsto-\nu_j$), results in a total of $2N$ linear equations for the $2N$ unknowns $(g_0^{(1)}(k^2,t), \ldots, g_0^{(N-1)}(k^2,t), g_1^{(0)}(k^2,t), \ldots, g_1^{(N)} (k^2,t))$. We write this system of equations in matrix form as 

\beq
\mathcal A(k) X(k^2,t) =  Y(k)-e^{k^2t} \mathcal Y(k, t),
\eeq

\no where $\mathcal A(k)$ is the coefficient matrix corresponding to the global relations \rf{eqn:globalrelations}, and 
\begin{align}
X(k^2,t) &= \left( ik g_0^{(1)}(k^2,t), \ldots, i kg_0^{(N-1)}(k^2,t), \sigma_0^2 g_1^{(0)}(k^2,t), \ldots, \sigma_{N}^2 g_1^{(N)} (k^2,t) \right)^\top\!\!\!\!,\\
Y(k) &= \left(\hat q_0^{(1)}(\nu_1), \ldots, \hat q_0^{(N)}(\nu_N), \hat q_0^{(1)}(-\nu_1), \ldots, \hat q_0^{(N)}(-\nu_N)\right)^\top\!\!\!\!,\\
\mathcal Y(k,t) &= \left(\hat q^{(1)}\left(\nu_1,t\right),\ldots, \hat q^{(N)}\left(\nu_N,t\right), \hat q^{(1)}\left(-\nu_1,t\right), \ldots, \hat q^{(N)}\left(-\nu_N,t\right)\right)^\top\!\!\!\!.
\end{align}

Following \cite{interface_maps}, we can show that the contribution to the solution of $\mathcal Y$ vanishes, so that, in effect, we may solve $\mathcal A X = Y$, for the unknown functions $g_m^{(j)}$. Using Cramer's rule, 

\beq
X_j = ikg_0^{(j)}(k^2,t) = \frac{\det\big(\mathcal A_j(k)\big)}{\det\big(\mathcal A(k)\big)}, ~~~~j=1,\cdots, N-1,
\eeq

\no where the matrix $\mathcal A_j(k)$ is $\mathcal A(k)$ with the $j$th column replaced by $Y$. If we multiply this equation by $e^{-k^2t}$ and integrate over $\partial \Omega$, 
shown in Figure~\ref{fig:Omega}{\textcolor{blue}a}, we recover the solution at the interfaces $q^{(j)}(x_{j},t)$, $j=1, \ldots, N-1$, \cite{JC_fokas_book} obtaining 

\beq
q^{(j)}(x_{j},t) = - \frac{1}{\pi}\int_{\partial \Omega} \frac{\det(\mathcal A_j(k))}{\det(\mathcal A(k))} e^{-k^2t} dk ~~~~\Rightarrow~~~~ q(x,t) = \lim_{N\to\infty} q^{(j)}(x_{j},t) .
\eeq
\no It is possible to compute the solution of the full interface problem as in \cite{interface_heat, interface_schrodinger, interface_heat_ring, interface_maps}, and obtain the same limit from there. To obtain \rf{eqn:q_sol}, we proceed as follows.

We introduce 
$\Lambda_p^\pm = \sigma_{p+1} + (-1)^{\ell_p+\ell_{p+1}}\sigma_p$, with $\ell_p, \ell_{p+1}\in \{0,1\}$, using 
$\Lambda_p^{+}$ when $\ell_p= \ell_{p+1}$ and 
$\Lambda_p^-$ when $\ell_p\neq \ell_{p+1}$. We define 
\beq 
D_N(k) = \frac{i}{2}\det(\mathcal A(k)) \left(\prod_{p=1}^{N-1} 
\frac{1}{\Lambda_p^+} \right) = \sum_{\substack{\boldsymbol \ell \in \{0,1\}^{N}\\\ell_1=0}} \left(\prod_{p=1}^{N-1} \frac{\Lambda_p^\pm}{\Lambda_p^+} \right) \sin\left(k\sum_{p=1}^{N} \frac{(-1)^{\ell_p}\Delta x_p}{\sigma_p}\right),
\eeq
where we have used the explicit form of $\mathcal A(k)$, see \cite{interface_maps} with slight modifications. We can show that

\beq
\label{eqn:prod_asymptotics}
\prod_{p=\ell}^{m-1} \frac{\Lambda_p^+}{2\sigma_p} = \sqrt{\frac{\sigma_m} {\sigma_\ell}} + \mathcal O\left(L\right), ~~~~\mbox{and}~~~~
\frac{\Lambda_p^-}{\Lambda_p^+} = \frac{\sigma_{p+1}-\sigma_p}{\sigma_{p+1}+\sigma_p} = \frac{\sigma'(x_p)\Delta x_p}{2\sigma_p} + \mathcal O \left(L^2\right),
\eeq
\no as $N\to\infty$ and $\Delta x_p\to 0^+$, with $L=\max_p \Delta x_p$. Next, we show 
\beq
\label{eqn:Delta}
\Delta (k) = \lim_{N\to\infty} D_N(k) = 
\sum_{n=0}^\infty \mathcal S_n^{(0,1)}(k),
\eeq
\no with 
$\mathcal S_n^{(a,b)}(k)$ 
defined in
\rf{eqn:Sn}.
Let $s_1, \ldots, s_n$ be the $n$ locations where $N$-dimensional vector $\boldsymbol \ell$ has its entry switch values (from 0 to 1 or from 1 to 0), with $s_0=0$ and $s_{n+1}=N$, for convenience. Since $\Lambda_p^\pm = \Lambda_p^+$ except where the switches occur, 
\beq
D_N(k) = \sum_{n=0}^{N-1} \sum_{0=s_0\leq s_1\leq \cdots \leq s_{n+1}=N} \left( \prod_{p=1}^{n} \frac{\Lambda_{s_p}^-}{\Lambda_{s_p}^+} \right) \sin\left(k \sum_{p=0}^n (-1)^p \sum_{r=1+s_p}^{s_{p+1}} \frac{\Delta x_r}{\sigma_r} \right).
\eeq
\no Using \rf{eqn:prod_asymptotics}, \rf{eqn:Delta} follows.

Turning to the numerator, we define
\beq
\label{eqn:Phi}
E_N(k,j) = \frac{1}{2}\det(\mathcal A_j(k))\left(\prod_{p=1}^{N-1} \frac{1}{\Lambda_p^+}\right)=\sum_{m=1}^{N}\left[ Y_{m}C_{m,j}+ Y_{m+N} \tilde C_{m,j}\right], 
\eeq
\no after doing a cofactor expansion using the $j$th column of $\mathcal A_j$. Here $C_{m,j}$ and $\tilde C_{m,j}$ are the relevant scaled cofactors of $\mathcal A_j$. Using
\beq
Y_{m} = e^{-i\nu_mx_m} q_0(x_m) \, \Delta x_m + \mathcal O\left(L^2\right) \qquad \text{ and } \qquad Y_{m+N} = e^{i\nu_mx_m}q_0(x_m) \, \Delta x_m+ \mathcal O\left(L^2\right) ,
\eeq
\no as $N\to\infty$, we have
\beq
E_N(k,j) = \sum_{m=1}^N q_0(x_m) \left( e^{-i\nu_mx_m} C_{m,j} +e^{i\nu_mx_m} \tilde C_{m,j} \right) \Delta x_m  + \mathcal O(L).
\eeq
Defining
\beq
\Psi_{j,m}=\sqrt{\sigma_j \sigma_m} \left( e^{-i\nu_m x_m} C_{m,j} + e^{i \nu_mx_m }\tilde C_{m,j}\right),
\eeq
then
\beq
\Phi(k,x) = \lim_{N\to\infty} E_N(k,j) = \lim_{N\to\infty} \sum_{m=1}^N \frac{\Psi_{j,m}}{\sqrt{\sigma_j\sigma_m}} \Delta x_m = \int_{0}^{1} \frac{\Psi(k,x,y) q_0(y)}{\sqrt{\sigma(x)\sigma(y)}} \, dy,
\eeq
\no where $x_j \to x$ and $x_m\to y$, and $\Psi(k,x,y) = \lim_{N\to\infty} \Psi_{j,m}$. For $1\leq m < j \leq N$,
\begin{align} \nonumber
\Psi_{j,m} &= \sqrt{\frac{\sigma_j}{\sigma_m}} \left(\prod_{p=m}^{j} \frac{2\sigma_p}{\Lambda_p^+}\right) \left[ \sum_{\substack{\boldsymbol \ell = \{0,1\}^m \\\ell_1=0}} \left(\prod_{p=1}^{m-1} \frac{\Lambda_p^\pm}{\Lambda_p^+} \right) \sin\left(k\sum_{p=1}^{m} \frac{(-1)^{\ell_p}\Delta x_p}{\sigma_p} \right) \right]
\times\\ & \hspace{1in} \times 
\left[ \sum_{\substack{\boldsymbol \ell = \{0,1\}^{N-j} \\ \ell_{j+1}=0}} \left(\prod_{p=j+1}^{N-1} \frac{\Lambda_p^\pm}{\Lambda_p^+} \right) \sin\left(k\sum_{p=j+1}^{N} \frac{(-1)^{\ell_p}\Delta x_p}{\sigma_p} \right) \right],
\end{align}

\no Again using the explicit form of $\mathcal A(k)$ \cite{interface_maps}. Taking the limit as before in \rf{eqn:Delta}, we find for $0< y < x < 1$,
\beq
\Psi(k,x,y) = \lim_{N\to\infty} \Psi_{j,m} = \left( \sum_{n=0}^\infty \mathcal S_{n}^{(0,y)}(k) \right)\left( \sum_{n=0}^\infty \mathcal S_n^{(x,1)}(k) \right) = \sum_{n=0}^\infty \sum_{\ell=0}^n \mathcal S_{n-\ell}^{(0,y)}(k)\mathcal S_\ell^{(x,1)}(k),
\eeq
\no and similarly, for $0<x<y<1$, $\Psi(k,x,y) = \Psi(k,y,x)$.

We can prove that the solution is well defined and solves the IBVP \rf{eqn:IBVP}, see \cite{farkasdeconinck2}. 




\section{Eigenvalue problem}

Consider the Sturm-Liouville problem 
\beq \label{eqn:eigenproblem}
\left( \sigma^2(x) y'\right)' = \lambda y,~~~~y(0)=0=y(1). 
\eeq
The eigenvalues $\lambda_m=-\kappa_m^2$ of this problem are related to the zeros $\kappa_m$ ($m=1, 2, \ldots$) of $\Delta(k)$. Since these eigenvalues are negative, it follows that these zeros are real. Since $\Delta(k)$ is odd, it suffices to only consider the positive zeros.

\begin{theorem}
\no The problem \rf{eqn:eigenproblem} has the eigenfunctions 
\beq \label{eqn:efuns}
X_m(x) = \frac{1}{\sqrt{\sigma(x)}} \sum_{n=0}^\infty \mathcal S_n^{(0,x)}(\kappa_n), ~~m=1, 2, \ldots \eeq
\end{theorem}
\begin{proof}
The proof is straightforward differentiation, noting the absolute convergence of the sums.
\end{proof}

\section{Numerical example}

\label{sec:numerics}
With $q_0(x) = x(1-x)$ and $\sigma^2(x) = (3-(2x-1)^2)/24$, we have the exact solution $q(x,t) = x(1-x)e^{-t}$. We construct an approximation to the solution \rf{eqn:q_sol} such that $q_N(x,t) \to q(x,t)$ as $N\to\infty$ (the index $N$ does not denote differentiation):
\beq \label{eqn:qapprox}
q_N(x,t) = \frac{1}{i\pi} \int_{\gamma} \frac{\exp\left(ik \int_0^1 \frac{d\xi}{\sigma(\xi)}\right)\Phi_N(k,x)}{\exp\left(ik \int_0^1 \frac{d\xi}{\sigma(\xi)}\right) \Delta_N(k)} e^{-k^2t} \, dk, 
\eeq
\no where we multiply denominator and numerator by the exponential so that both are decaying in the upper-half complex $k$ plane, and where we truncate each series up to $n=N$. The contour $\gamma$ is used instead of $\partial \Omega$ to aid convergence as the factor $\exp(-k^2 t)$ decays along it. The results are shown in Figure~\ref{fig:heat}. 
\begin{figure}[tb]
\begin{center}
\def \sc {0.53}
\begin{tabular}{cc}
\includegraphics[scale=\sc]{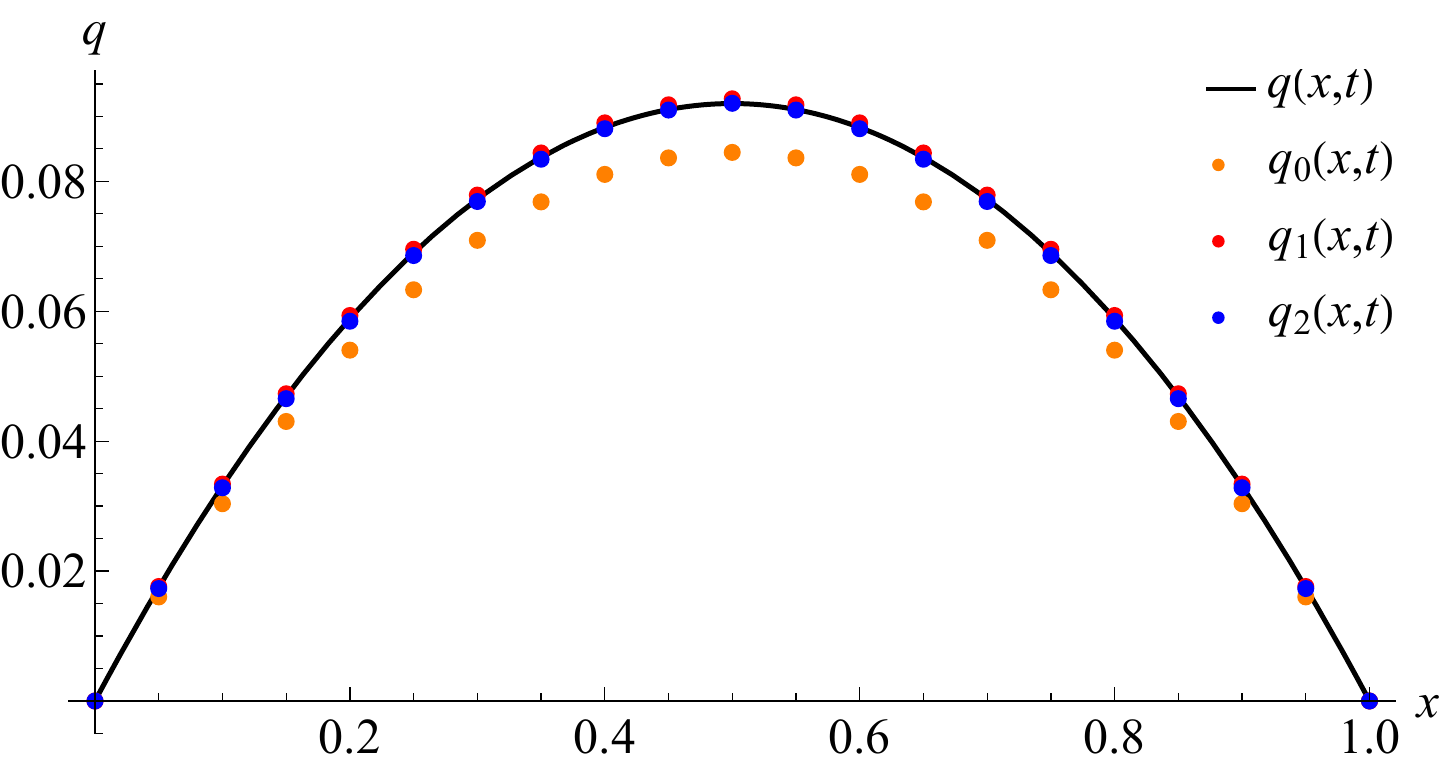} & \includegraphics[scale=\sc]{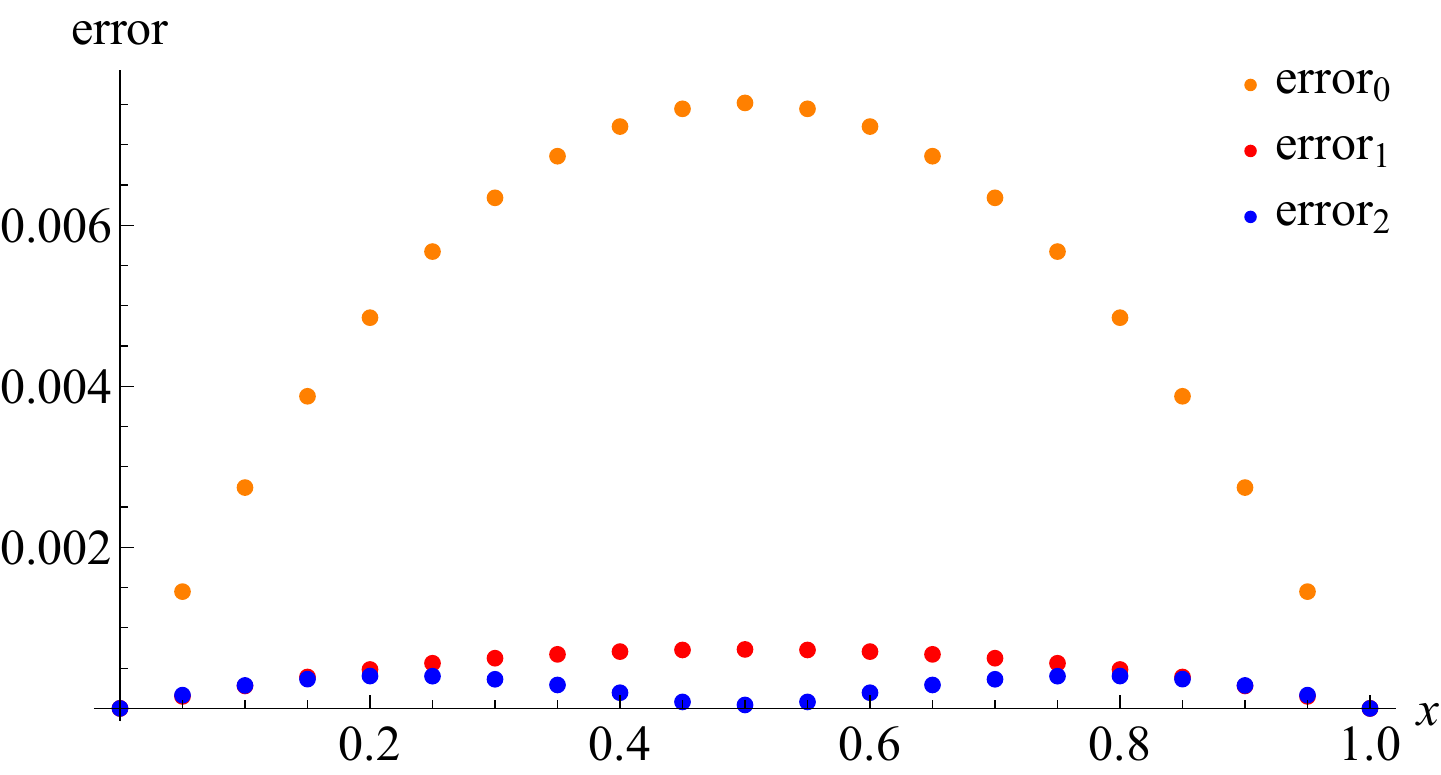}\\
(a) & (b)
\end{tabular}
\end{center}
\caption{(a) The exact solution $q(x,t) = x(1-x)e^{-t}$ and its successive approximations \rf{eqn:qapprox} for $N=0,1,2$. (b) The error of the approximations. }
\la{fig:heat}
\end{figure}

To demonstrate the computation of the eigenvalues and eigenfunctions, we consider \rf{eqn:eigenproblem} with the same $\sigma(x)$. To find the eigenvalues, we use Mathematica's FindRoot command on $\Delta_N(k)$. The results are shown in Table~\ref{tb:eigenvalues}. We see that our method converges to the eigenvalues and outperforms Mathematica's built-in NDEigenvalues command for $n=4$. Futhermore, we are able to provide explicit bounds on the eigenvalue approximations \cite{farkasdeconinck2}. Lastly, we denote the order-$N$ truncated eigenfunctions \rf{eqn:efuns} as $X_m^{(N)}(x)$. These are shown in Figure~\ref{fig:efuns}. For the simple $\sigma(x)$, given above, the order-0 truncation is quite accurate. For a more complicated $\sigma(x)$, the order-1 truncation gives an accurate representation.

\begin{table}[tb]
\begin{center}
\begin{tabular}{|c|c|c|c|c|} \hline
Method: & $\lambda_1$ & $\lambda_2$ & $\lambda_3$ & $\lambda_4$ \\ \hline
chebfun & $-1.0000$ & $-4.2540$ & $-9.6812$ & $-17.2800$ \\
NDEigenvalues & $-1.0000$ & $-4.2540$ & $-9.6818$ & $-17.2834$ \\
FindRoot: $\Delta_0(k)$ & $-1.0856$ & $-4.3423$ & $-9.7702$ & $-17.3692$ \\
FindRoot: $\Delta_1(k)$ & $-0.9917$ & $-4.2474$ & $-9.6749$ & $-17.2737$ \\
FindRoot: $\Delta_2(k)$ & $-1.0006$ & $-4.2542$ & $-9.6814$ & $-17.2801$ \\ \hline
\end{tabular}
\caption{\label{tb:eigenvalues}Eigenvalues of the system \rf{eqn:eigenproblem} calculated using MATLAB's chebfun package compared to using Mathematica's NDEigenvalues and Mathematica's FindRoot on $\Delta_N(k)$ for $N=0,1,2$.}
\end{center}
\end{table}
\begin{figure}[tb]
\begin{center}
\def \sc {0.53}
\begin{tabular}{cc}
\includegraphics[scale=\sc]{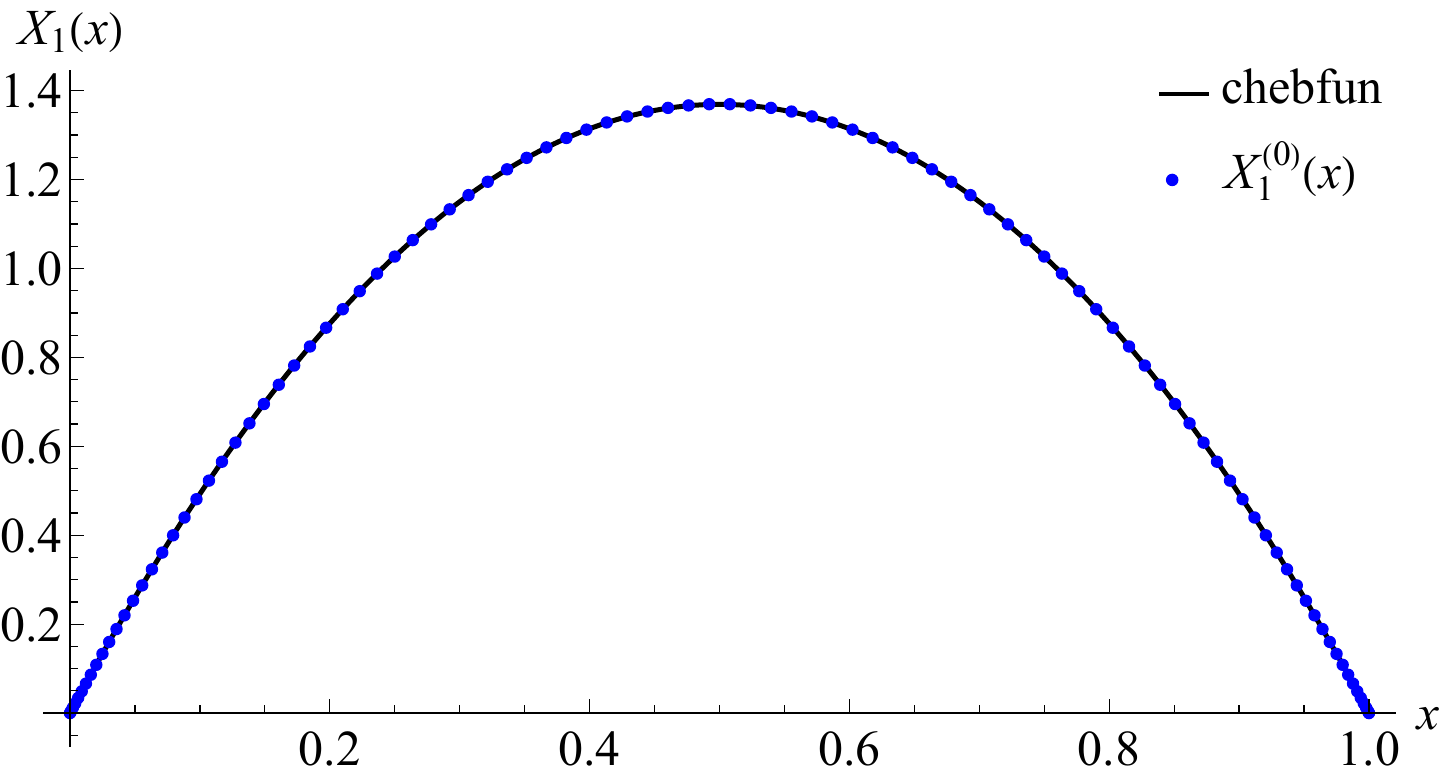} & \includegraphics[scale=\sc]{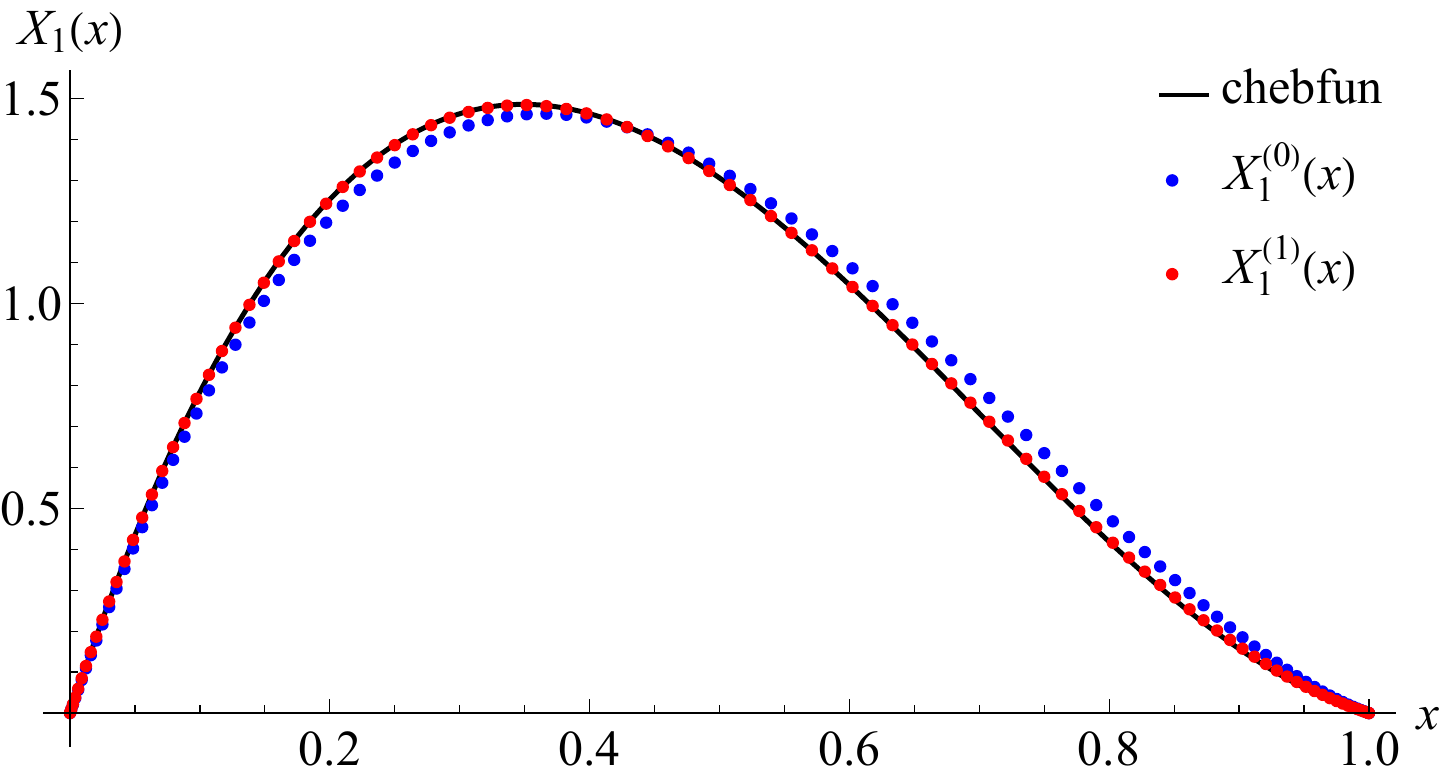}\\
(a) & (b)
\end{tabular}
\end{center}
\caption{(a) The successive approximations to the eigenfunctions \rf{eqn:efuns} shown with the numerically computed eigenfunctions using MATLAB's chebfun package. (b) The same, using $\sigma^2(x) = \frac{6337-252\sqrt{111}-4500x^2(11+x(-14+5x))}{9000(11+6x(-7+5x))}$. }
\la{fig:efuns}
\end{figure}











{\small
\bibliographystyle{abbrv}
\bibliography{references}
}

\end{document}